\documentclass{amsart}
\usepackage[utf8]{inputenc}
\usepackage{amssymb,latexsym}
\usepackage{amsmath}
\usepackage{graphicx}
\usepackage{textcomp}

\usepackage{amsthm,amssymb,enumerate,graphicx, tikz}
\usepackage{amscd}
\usepackage{setspace}
\usepackage{comment}
\usepackage{hyperref}
\usepackage{cleveref}

\newtheorem{theorem}{Theorem}[section]

\newtheorem*{theorem*}{Theorem}

\newtheorem{lemma}[theorem]{Lemma}

\newtheorem{corollary}[theorem]{Corollary}

\newtheorem{conjecture}[theorem]{Conjecture}
\theoremstyle{definition}
\newtheorem{definition}[theorem]{Definition}

\theoremstyle{remark}

\newcommand{\F}{\mathcal{F}}

\newcommand{\conv}{\textrm{conv}}

\newcommand{\R}{\mathbb{R}}

\title{Piercing families of convex sets in the plane that avoid a certain subfamily with lines}
\author{Daniel McGinnis}
\thanks{The author graciously acknowledges support from the Iowa State University Department of Mathematics through the Lambert Graduate Assistantship}
\date{}

\begin{document}

\maketitle

\begin{abstract}
     We define a $C(k)$ to be a family of $k$ sets $F_1,\dots,F_k$ such that $\conv(F_i\cup F_{i+1})\cap \conv(F_j\cup F_{j+1})=\emptyset$ when $\{i,i+1\}\cap \{j,j+1\}=\emptyset$ (indices are taken modulo $k$). We show that if $\F$ is a family of compact, convex sets that does not contain a $C(k)$, then there are $k-2$ lines that pierce $\F$. Additionally, we give an example of a family of compact, convex sets that contains no $C(k)$ and cannot be pierced by $\left\lceil \frac{k}{2} \right\rceil -1$ lines.
\end{abstract}

\section{Introduction}
Let $\F$ be a family of sets in the plane; $\F$ is said to have a \textit{line transversal} if there is a line that intersects each set in $\F$. If every $r$ sets in $\F$ have a line transversal, then $\F$ is said to have the $T(r)$-property, and $\F$ is said to be $T^n$ if there are $n$ lines whose union intersects each set in $\F$. In this case we say $\F$ is \textit{pierced} by these lines. In 1969, Eckhoff showed that if $\F$ is a family of compact, convex sets that has the $T(r)$-property where $r\geq 4$, then $\F$ is $T^2$ \cite{eckhoff1969}. A result of Santalo shows that this result is best possible \cite{santalotheorem1940}, i.e. for all $r$, there exists a family of compact, convex sets with the $T(r)$ property that does not have a line transversal. Eckhoff also showed in 1973 that the $T(3)$-property does not imply $T^2$ \cite{eckhofftransversal1973}. In 1975, Kramer proved that the $T(3)$-property implies $T^5$ \cite{Kramer}. Eckhoff later showed in 1993 that the $T(3)$-property implies $T^4$, and conjectured that the $T(3)$-property in fact implies $T^3$ \cite{eckhoffgallai}. This conjecture has recently been verified by McGinnis and Zerbib \cite{mcginnis2021line}. In fact, they proved a stronger statement, which we now explain.

Three sets $F_1,F_2,F_3$ in the plane are said to be a \textit{tight triple} if $\conv(F_1\cup F_2)\cap \conv(F_2\cup F_3)\cap \conv(F_3\cup F_1)\neq \emptyset$. This was first defined by Holmsen \cite{holmsenGeometric2013}. A family of planar sets will be called \textit{a family of tight triples} if every three sets in the family are a tight triple. If three sets have a line transversal, then they are a tight triple as the convex hull of two of the three sets intersects the third. McGinnis and Zerbib showed that a family of tight triples consisting of compact, convex sets is $T^3$, which implies Eckhoff's conjecture. 



The main purpose of this paper is to prove an extension of the result that families of tight triples are $T^3$. We define a certain type of family of sets, which we call a $C(k)$.

\begin{definition}
For $k\geq 4$, we define a $C(k)$ to be a family of $k$ distinct sets in the plane together with a linear ordering, say $F_1,\dots,F_k$ where the sets are ordered by their indices, such that $\conv(F_i\cup F_{i+1})\cap \conv(F_j\cup F_{j+1})=\emptyset$ when $\{i,i+1\}\cap \{j,j+1\}=\emptyset$ (indices are taken modulo $k$). Additionally, we define a $C(3)$ to be a family of three disjoint sets in the plane that is not a tight triple.
\end{definition}

Roughly speaking, if $F_1,\dots,F_k$ is a $C(k)$, then the union $\cup_{i=1}^{k}\conv(F_i\cup F_{i+1})$ resembles a closed loop that does not cross itself (see Figure \ref{fig:C(5)}). Notice that the sets in a $C(k)$ are pairwise disjoint.

\begin{figure}
    \centering
    \includegraphics[scale=.4]{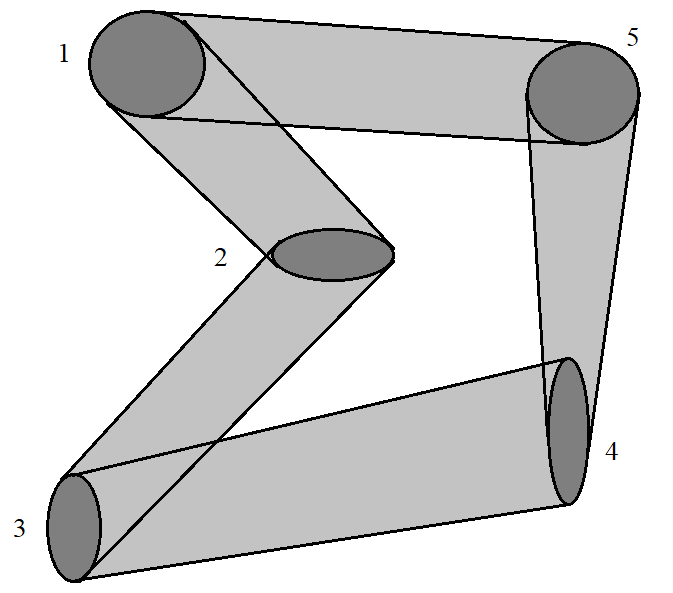}
    \caption{An example of a $C(5)$. The sets of the $C(k)$ are in dark gray, and the ordering on these sets is indicated by the numbers $1,\dots,5$. The dark gray and light gray together depict the union $\bigcup_{i=1}^5\conv(F_i\cup F_{i+1})$.}
    \label{fig:C(5)}
\end{figure}

A family $\F$ is said to be \textit{$C(k)$-free} if $\F$ does not contain a $C(k)$ as a subfamily. We note that a $C(k)$-free family may not be $C(k-1)$-free, and similarly, a $C(k-1)$-free family need not be $C(k)$-free. (see Figure \ref{fig:C(5)-free}).

\begin{figure}
    \centering
    \includegraphics[scale=.6]{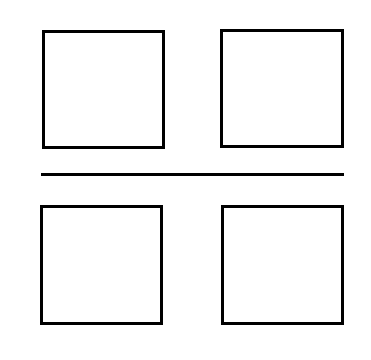}
    \caption{A $C(5)$-free and $C(3)$-free family that is not $C(4)$-free.}
    \label{fig:C(5)-free}
\end{figure}

Let $L(k)$ be the smallest integer such that any $C(k)$-free family of compact, convex sets can be pierced by $L(k)$ lines. The following is the main result of this paper.

\begin{theorem}\label{main}
Let $k\geq 4$. We have the following:
\[
\left\lceil \frac{k}{2} \right\rceil\leq L(k)\leq k-2.
\]
\end{theorem}

For $k=4$, the lower bound of Theorem \ref{main} follows from the result of Santalo \cite{santalotheorem1940} that there are families with the $T(4)$-property that do not have a line transversal. This is due to the fact that a $C(k)$ cannot have a line transversal, so a family with the $T(4)$-property is in particular $C(4)$-free. We also note that the upper bound for $k=4$ was essentially proved in the concluding remarks of \cite{mcginnis2021line}. Indeed, the proof outlined in \cite{mcginnis2021line} shows that if $\F$ is a family of compact, convex sets in the plane that is not $T^2$, then there are non-parallel lines such that each quadrant defined by these two lines contains a set from $\F$. These 4 sets then make up a $C(4)$.

For $k=5$, we get a tight result.
\begin{corollary}
The following equality holds: $$L(5)=3.$$
\end{corollary}

The proof of Theorem \ref{main} is split into two sections. Section \ref{lower} is dedicated to proving the lower bound of Theorem \ref{main}, and Section \ref{upper} is dedicated to the upper bound. For 2 points $p$ and $q$, we denote by $[p,q]$ to be the line segment connecting $p$ and $q$. If $p=q$, then $[p,q]$ consists of a single point.

\section{The lower bound}\label{lower}

In this section we exhibit a $C(k)$-free family $\F$ that is not $T^{\lceil \frac{k}{2} \rceil-1}$. The inspiration for the construction of such a family comes from \cite{eckhofftransversal1973} where an example of a family of compact, convex sets with the $T(3)$-property that is not $T^2$ is exhibited by Eckhoff. As mentioned earlier, the result is already established for $k=4$ so we may assume that $k\geq 5$. 

For $k$ odd, we will present a family $\F$ that is both $C(k)$-free and $C(k+1)$-free and is not $T^{\lceil \frac{k}{2} \rceil-1}$. This will establish the lower bound of Theorem \ref{main}. We note that for $k$ even, an example of a $C(k)$-free family that is not $T^{\lceil \frac{k}{2} \rceil-1}$ is given simply by $k-1$ points in general position. However, in this example, the family is $C(k)$-free for the seemingly trivial reason that are only $k-1$ sets in the family. For each $k\geq 5$, we present a family demonstrating the lower bound of Theorem \ref{main} that contains more than $k$ sets.

Let $p_1,\dots,p_{3(k-1)}$ be equidistant points on the unit circle, arranged clockwise. For $1\leq i\leq 3$, let $p^\ell_i$ be a point lying slightly counterclockwise to $p_i$ and $p^r_i$ to be a point lying slightly clockwise to $p_i$ in such a way that the points $p^\ell_1,\,p^r_1,\,p^\ell_2,\,p^r_2,\,p^\ell_3,\,p^r_3,\,p_4,\,\dots\,,p_{3(k-1)}$ are arranged in clockwise order.

We first define three families of sets, which consists only of line segments (see Figure \ref{fig:example}).
\[
\F_1=\{[p_1^r,p_3^\ell],[p_4,p_6],[p_7,p_9],\dots,[p_{3(k-1)-2},p_{3(k-1)}]\}
\]
\[
\F_2=\{[p_2^r,p_4],[p_5,p_7],[p_8,p_{10}],\dots,[p_{3(k-1)-4},p_{3(k-1)-2}],[p_{3(k-1)-1},p_1^\ell]\}
\]
\[
\F_3=\{[p_3^r,p_5],[p_6,p_8],[p_{10},p_{12}],\dots,[p_{3(k-1)-3},p_{3(k-1)-1}],[p_{3(k-1)},p_2^\ell]\}.
\]
Finally, we take $\F=\F_1\cup \F_2\cup \F_3$ (see Figure \ref{fig:example}). We now show that $\F$ is $C(k)$-free and $C(k+1)$-free, and we show it is not $T^{\lceil \frac{k}{2} \rceil-1}$. For a set $[p,q]\in \F$ we say that $p$ comes clockwise before $q$ if the clockwise distance along the unit circle from $p$ to $q$ is less than the clockwise distance from $q$ to $p$. When $[p,q]\in \F$ and $p$ comes clockwise before $q$, we denote arc$[p,q]$ to be the arc along the unit circle that goes clockwise from $p$ to $q$. Also, we use $I([p,q])$ to denote the set of indices $i$ such that $p_i\in$ arc$[p,q]$, $p_i^\ell\in$ arc$[p,q]$, or $p_i^r\in$ arc$[p,q]$. For example, $I(p_2^r,p_4)=\{2,3,4\}$. Note that if $F_1,F_2\in \F$ intersect, then $I(F_1)\cap I(F_2)\neq \emptyset$.

\begin{figure}
    \centering
    \begin{tikzpicture}
    
    \draw[red] (4*360/12-5:3cm)--(2*360/12+5:3cm);
    \draw[red] (360/12:3cm)--(11*360/12:3cm);
    \draw[red] (10*360/12:3cm)--(8*360/12:3cm);
    \draw[red] (7*360/12:3cm)--(5*360/12:3cm);
    
    \draw[blue] (3*360/12-5:3cm)--(360/12:3cm);
    \draw[blue] (3,0)--(10*360/12:3cm);
    \draw[blue] (9*360/12:3cm)--(7*360/12:3cm);
    \draw[blue] (6*360/12:3cm)--(4*360/12+5:3cm);
    
    \draw[green] (2*360/12-5:3cm)--(3,0);
    \draw[green] (11*360/12:3cm)--(9*360/12:3cm);
    \draw[green] (8*360/12:3cm)--(6*360/12:3cm);
    \draw[green] (5*360/12:3cm)--(3*360/12+5:3cm);
    
    \draw (0,0) circle (3cm);
    \filldraw [black] (3,0) circle (2pt) node[right] {$p_5$};
    \filldraw [black] (360/12:3cm) circle (2pt) node[right] {$p_4$};
    \filldraw [black] (2*360/12+5:3cm) circle (2pt) node[above] {$p_3^\ell$};
    \filldraw [black] (2*360/12-5:3cm) circle (2pt) node[above] {$p_3^r$};
    \filldraw [black] (3*360/12-5:3cm) circle (2pt) node[above] {$p_2^r$};
    \filldraw [black] (3*360/12+5:3cm) circle (2pt) node[above] {$p_2^\ell$};
    \filldraw [black] (4*360/12-5:3cm) circle (2pt) node[above] {$p_1^r$};
    \filldraw [black] (4*360/12+5:3cm) circle (2pt) node[above] {$p_1^\ell$};
    \filldraw [black] (5*360/12:3cm) circle (2pt) node[left] {$p_{12}$};
    \filldraw [black] (6*360/12:3cm) circle (2pt) node[left] {$p_{11}$};
    \filldraw [black] (7*360/12:3cm) circle (2pt) node[left] {$p_{10}$};
    \filldraw [black] (8*360/12:3cm) circle (2pt) node[below] {$p_9$};
    \filldraw [black] (9*360/12:3cm) circle (2pt) node[below] {$p_8$};
    \filldraw [black] (10*360/12:3cm) circle (2pt) node[below] {$p_7$};
    \filldraw [black] (11*360/12:3cm) circle (2pt) node[right] {$p_6$};
    
    \end{tikzpicture}
    \caption{An example demonstrating the lower bound of Theorem \ref{main} for $k=5$. The sets in $\F_1$ are red, in $\F_2$ are blue, and in $\F_3$ are green.}
    \label{fig:example}
\end{figure}

\begin{lemma}\label{freeness}
The family $\F$ is $C(k)$-free and $C(k+1)$-free.
\end{lemma}
\begin{proof}
Let $[p,q]\in \F$ where $p$ comes clockwise before $q$. There is no set of $\F$ that is disjoint from $[p,q]$ and contains a point on arc$[p,q]$. Since for each such set $[p,q]$, arc$[p,q]$ contains at least 3 of the $3k$ points in
\[
P=\{p^\ell_1,\,p^r_1,\,p^\ell_2,\,p^r_2,\,p^\ell_3,\,p^r_3,\,p_4,\,\dots\,,p_{3(k-1)}\},
\]
a $C(k)$ of $\F$ has the property that each set $[p,q]$ in the $C(k)$ contains exactly 3 points in arc$[p,q]$, and every point of $P$ is in arc$[p,q]$ for some $[p,q]$ of the $C(k)$. However, each set $[p',q']$ of $\F$ that contains $p_2^r$ in arc$[p',q']$ has the property that at least 4 points of $P$ are contained in arc$[p',q']$, a contradiction. Also, $\F$ is $C(k+1)$-free by the same reasoning.
\end{proof}

\begin{lemma}
The family $\F$ is not $T^{\lceil \frac{k}{2} \rceil-1}$.
\end{lemma}
\begin{proof}
Notice that for any $F\in \F$, if a line $L$ pierces $F$, then $L$ intersects arc$F$. Any point on the unit circle is contained in arc$F$ for at most 3 sets $F\in \F$. If a point is contained in arc$F$ for 3 such sets $F\in \F$, then this point must be of the form $p_i$ for some $4\leq i\leq 3(k-1)$. Since a line intersects the unit circle in at most 2 points, any line intersects at most 6 sets in $\F$. 

Now, since  arc$[p_1^r,p_3^\ell]$ does not contain $p_i$ for any $4\leq i\leq 3(k-1)$, there is no line that intersects $[p_1^r,p_3^\ell]$ and intersects 6 sets in $\F$.

It follows that if $a$ lines pierce $\F$, then $6(a-1)+5\geq 3(k-1)$, so $a > \frac{k-1}{2} = \left\lceil \frac{k}{2} \right\rceil-1$. This completes the proof.
\end{proof}

\section{The upper bound}\label{upper}

In this section, we prove that every $C(k)$-free family $\F$ can be pierced by $k-2$ lines, and again, we may assume that $k\geq 5$.
Because the sets of $\F$ are compact, it is the case that if every finite subfamily of $\F$ is $T^n$, then $\F$ is $T^n$. This is stated for instance in \cite{eckhoffgallai}. Therefore, throughout this section we may assume that $\F$ is finite, and thus, we may scale the plane so that each set of $\F$ is contained in the open unit disk. 

First, we will need to introduce a topological tool known as the KKM Theorem \cite{knaster1929}.

Let $\Delta^{n-1} = \{(x_1,\dots,x_n) \in \R^n \mid x_i\ge 0, \sum_{i=1}^n x_i = 1\}$ denote the $(n-1)$-dimensional simplex in $\R^n$, whose vertices are the canonical basis vectors $e_1,\dots,e_n$. A face $\sigma$ of $\Delta^{n-1}$ is a subset of $\Delta^{n-1}$ of the form $\conv(\{e_i : i\in I\})$ for some $I\subset [n]$.
\begin{theorem}
Let $A_1,\dots,A_{n}$ be open sets such that for every face $\sigma$ of $\Delta^{n-1}$ we have $\sigma \subset \bigcup_{e_j \in \sigma} A_j$.  Then we have that $\cap_{i=1}^{n} A_{i}\neq \emptyset$. 
\end{theorem}

We remark that the original KKM Theorem was stated for when the sets $A_i$ are closed, and the statement where the $A_i$'s are open as stated here appears in e.g. \cite{openkkm}.

Let $U$ be the unit circle, and let $f:[0,1] \rightarrow U$ be a parameterization of $U$ defined by $f(t)=(\textrm{cos}(2\pi t), \textrm{sin}(2\pi t))$. 

A point $x=(x_1,\dots,x_{2(k-2)})\in \Delta^{2(k-2)-1}$ corresponds to $2(k-2)$ points on $U$ given by $f_i(x)=f(\sum_{j=1}^i x_{j})$ for $0\leq i\leq 2(k-2)$ (note that $f_0(x)=f_{2(k-2)}(x)$). We define the line segments $\ell_i(x)=[f_i(x),f_{i+k-2}(x)]$ for $0\leq i\leq 2(k-2)-1$ where addition is take modulo $2(k-2)$ (see Figure \ref{fig:R1}). Note that $\ell_i(x)=\ell_{i+k-2}(x)$. 

For $1\leq i\leq k-3$, we define the region $R^i_x$ to be the open set bounded by the lines $\ell_{i}(x)$ and $\ell_{i-1}(x)$ and by the arc from $f_{i-1}(x)$ to $f_i(x)$. For $k-2\leq i\leq 2(k-2)$, we define $R^i_x$  to be the intersection of the region in the open unit disk bounded by $\ell_{i-1}(x)$, $\ell_{i}(x)$ and the arc from $f_{i-1}(x)$ to $f_i(x)$ with the open halfspaces defined by $\ell_j(x)$ for those $1\leq j\leq 2(k-2)$ for which $\ell_j(x)$ is a line segment (and not a point) containing the open arc from $f_{i-1}(x)$ to $f_i(x)$ (see Figures \ref{fig:R1} - \ref{fig:R6}). 

\begin{lemma}\label{extra_region}
Assume that each $R^i_x$ is non-empty. Let $Q_1$ be the open quadrant defined by the lines $\ell_{k-3}(x)$ and $\ell_{k-2}(x)$ that contains the open arc along the unit circle from $f_0(x)$ to $f_{k-3}(x)$. There is some $1\leq j\leq k-3$ such that $R^j_x$ is contained in $Q_1$.
\end{lemma}
\begin{proof}
Let $Q_2,Q_3,Q_4$ be the remaining quadrants defined by $\ell_{k-3}(x)$ and $\ell_{k-2}(x)$ so that $Q_1,Q_2,Q_3,Q_4$ occur in counterclockwise order (see Figure \ref{fig:quadrants}). Since the regions $R^i_x$ are non-empty, the $Q_i$'s are non-empty. Assume that $R^{k-3}_x$ is not contained in $Q_1$, then $\ell_{k-4}(x)$ intersects $Q_4$. Similarly, if we assume that $R^1_x$ is not contained in $Q_1$, then $\ell_1(x)$ intersects $Q_2$. If $k=5$, then we are done since $k-4=1$ and $\ell_1(x)$ cannot intersect both $Q_2$ and $Q_4$. Otherwise, $k\geq 6$, and we take $j$ to be the smallest index such that $\ell_j(x)$ intersects $Q_4$. Such an index exists since we assume that $\ell_{k-4}(x)$ intersects $Q_4$. Also, $j>1$, since we assume that $\ell_1(x)$ intersects $Q_2$ (and hence does not intersect $Q_4$). Therefore $j-1\geq 1$ and $\ell_{j-1}(x)$ intersects $Q_2$, or the intersection of $\ell_{k-2}(x)$ and $\ell_{k-3}(x)$. This implies that $\ell_{j-1}(x)$ and $\ell_{j}(x)$ intersect in $Q_1$ (see Figure \ref{fig:quadrants}). Therefore, $R^j_x$ is contained in $Q_1$. 
\end{proof}

\begin{lemma}\label{covers}
If a connected set $F$ contained in the unit disc does not intersect any $\ell_j(x)$, then $F$ is contained in $R^i_x$ for some $i$.
\end{lemma}
\begin{proof}
Let $\tilde{R}^i_x$ be the region bounded by the arc from $f_{i-1}(x)$ to $f_i(x)$ and by the lines $\ell_{i-1}(x)$ and $\ell_{i}(x)$. Note that $\tilde{R}^i_x=R^i_x$ for $1\leq i\leq k-3$. Also, we have that $F$ is contained in $\tilde{R}^i_x$ for some $i$. If $1\leq i\leq k-3$, then we are done since $\tilde{R}^i_x=R^i_x$. So assume that $i\geq k-2$ and $F$ is not contained in $R^i_x$. Since $F$ does not intersect any $\ell_j(x)$, there is some $j$ such that $F$ is contained in the open halfspace defined by $\ell_j(x)$ that does not contain the arc from $f_{i-1}(x)$ to $f_i(x)$. 

If $i=2(k-2)$, then choose the largest index $j\in \{1,\dots, k-4\}$ such that the open halfspace defined by $\ell_j(x)$ not containing the arc from $f_{2(k-2)-1}(x)$ to $f_0(x)$ contains $F$. Then $F$ is contained in $R^{j+1}_x$. This is because $R^{j+1}_x$ is the region in the open unit disk obtained by taking intersection of the open halfspace defined by $\ell_{j}(x)$ that does not contain the arc from $f_{2(k-2)-1}(x)$ to $f_0(x)$ (which contains $F$) with the open halfspace defined by $\ell_{j+1}(x)$ that contains the arc from $f_{2(k-2)-1}(x)$ to $f_0(x)$ (which contains $F$ by the maximality of $j$).

If $i=k-2$, then choose the smallest index $j\in \{1,\dots, k-4\}$ such that the halfspace defined by $\ell_j(x)$ not containing the arc from $f_{k-3}(x)$ to $f_{k-2}(x)$ contains $F$. Then $F$ is contained in $R^{j}_x$ (by similar reasoning as above).

Hence, we may assume that $i\neq k-2,2(k-2)$. Let $i'=i-(k-2)$. If there is some $u\in \{0,\dots,i'-2\}$ such that the halfspace defined by $\ell_u(x)$ not containing the arc from $f_{i-1}(x)$ to $f_{i}(x)$ contains $F$, then let $j$ be the largest such index. Then $F$ is contained in $R^{j+1}_x$. Otherwise, choose the smallest index $j\in \{k-3,\dots,i'+1\}$ such that  the halfspace defined by $\ell_j(x)$ not containing the arc from $f_{i-1}(x)$ to $f_{i}(x)$ contains $F$. Then $F$ is contained in $R^j_x$.

This completes the proof.
\end{proof}

With the goal of using the KKM Theorem, we define $A_i$ to be the set of points $x\in \Delta^{2(k-2)-1}$ for which $R^i_x$ contains a set in $\F$. Because $R^i_x$ is open and each set in $\F$ is closed, it follows that $A_i$ is open. Let us assume for contradiction that there is no point $x\in \Delta^{2(k-2)-1}$ for which the lines $\ell_j(x)$, $0\leq j\leq k-3$ pierce $\F$. Then by Lemma \ref{covers}, for each $x\in \Delta^{2(k-2)-1}$, there is some region $R^i_x$ that contains a set in $\F$. It follows that $\Delta^{2(k-2)-1}\subset \cup_{i=1}^{2(k-2)}A_i$. Also, it is clear that for $x=(x_1,\dots,x_{2(k-2)})$, if $x_i=0$, then the region $R^i_x$ is empty and hence $x\notin A_i$. It follows from this fact that the sets $A_i$ satisfy the conditions of the KKM Theorem.

Therefore, by the KKM Theorem, there exists a point $x\in \cap_{i=1}^{2(k-2)}A_i$. Notice in particular that each $R^i_x$ is non-empty.

Let $1\leq i\leq k-3$ be the index such that $R^i_x$ is contained in $Q_1$, guaranteed by Lemma \ref{extra_region} where $Q_1$ is defined as in Lemma \ref{extra_region}.

Let $F_1$ be the set in $\F$ contained in $R^i_x$, and let $F_j$ be the set of $\F$ contained in $R^{k-4+j}_x$ for $2\leq j\leq k$. Note that the corresponding regions $R^i_x$ and $R^j_x$ $2\leq j\leq k$ are disjoint, so $F_1,\dots,F_k$ are pairwise distinct. 

Now, $\conv(F_1\cup F_2)$ is separated from $F_3,\dots,F_k$ by the line $\ell_0(x)$, so $\conv(F_1\cup F_2)$ is disjoint from $\conv(F_u\cup F_{u+1})$ for all $3\leq u\leq k-1$. For $j\in \{2,\dots,k-1\}$, $\conv(F_j\cup F_{j+1})$ is separated from $F_{j+2},\dots, F_k$ by $\ell_{k-4+(j+1)}(x)$, so $\conv(F_j\cup F_{j+1})$ is disjoint from $\conv(F_u\cup F_{u+1})$ for all $j+2\leq u\leq k-1$. Finally, $\conv(F_k\cup F_{1})$ is separated from $F_2,\dots,f_{k-1}$ by $\ell_{k-3}(x)$, so $\conv(F_k\cup F_{1})$ is disjoint from $\conv(F_u\cup F_{u+1})$ for all $2\leq u\leq k-2$. It follows that the sets $F_1,\dots,F_k$ form a $C(k)$, a contradiction.

This completes the proof of Theorem \ref{main}. 

\section{Concluding Conjecture}

We present a conjecture for the correct value of $L(k)$, which states that the lower bound of Theorem \ref{main} is correct.

\begin{conjecture}
We have that $L(k)=\left\lceil \frac{k}{2} \right\rceil$.
\end{conjecture}

\section{Acknowledgements}

The author would like to thank Shira Zerbib for commenting on a first draft of this paper.


\begin{figure}
    \centering
    \includegraphics[scale=.25]{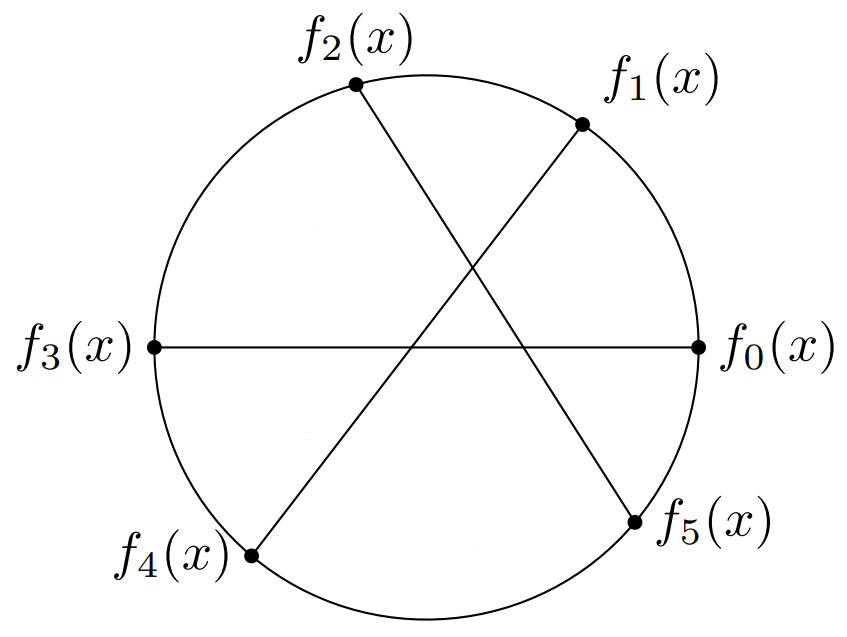}
    \includegraphics[scale=.25]{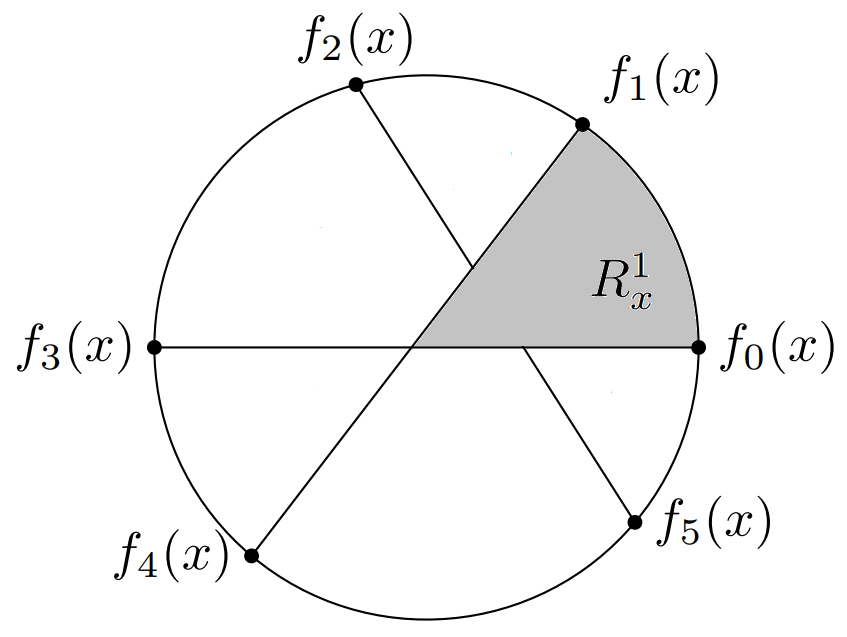}
    \caption{Three line segments corresponding to a point $x\in \Delta^5$ (on the left), occurring in the case $k=5$, and the region $R^1_x$ (on the right).}
    \label{fig:R1}
\end{figure}

\begin{figure}
    \centering
    \includegraphics[scale=.25]{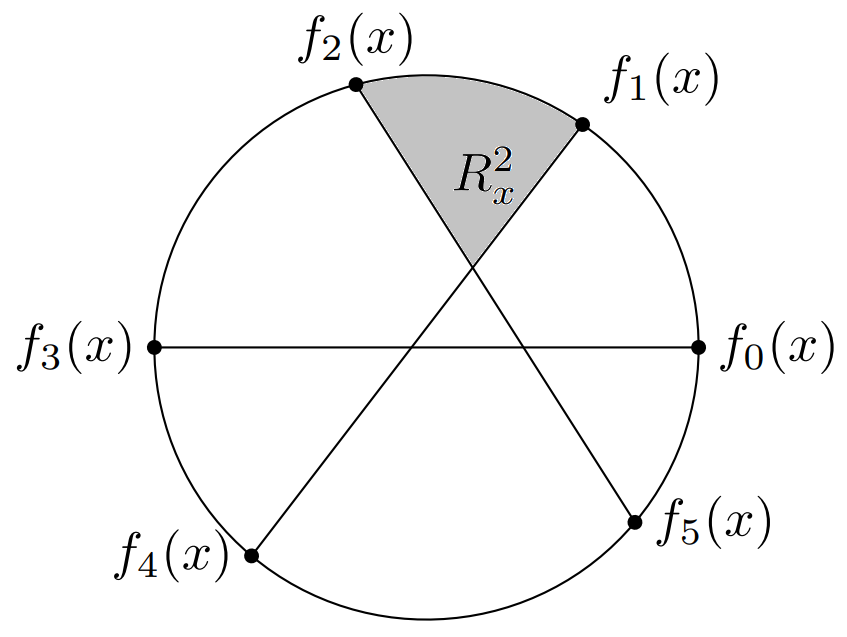}
    \includegraphics[scale=.25]{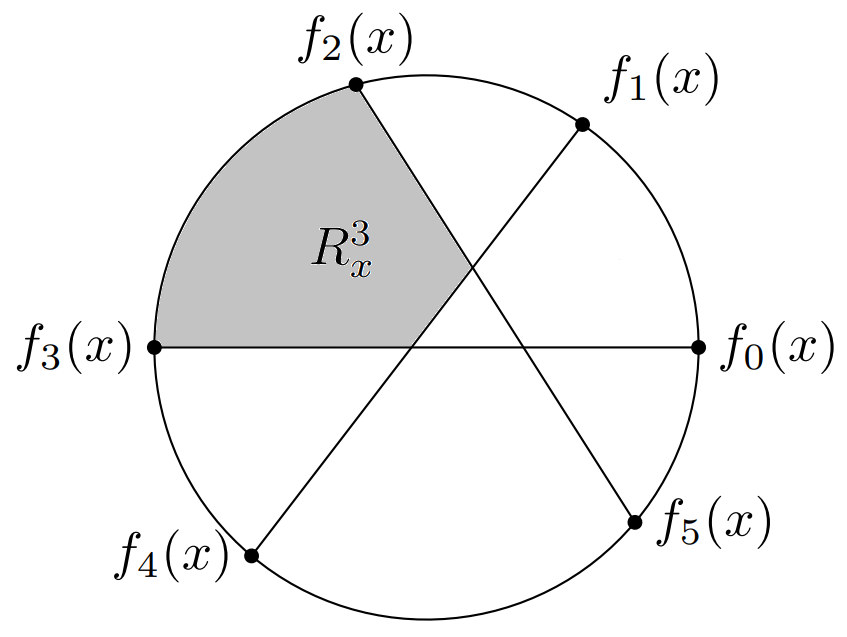}
    \caption{The regions $R^2_x$ (on the left) and $R^3_x$ (on the right).}
    \label{fig:R2R3}
\end{figure}

\begin{figure}
    \centering
    \includegraphics[scale=.25]{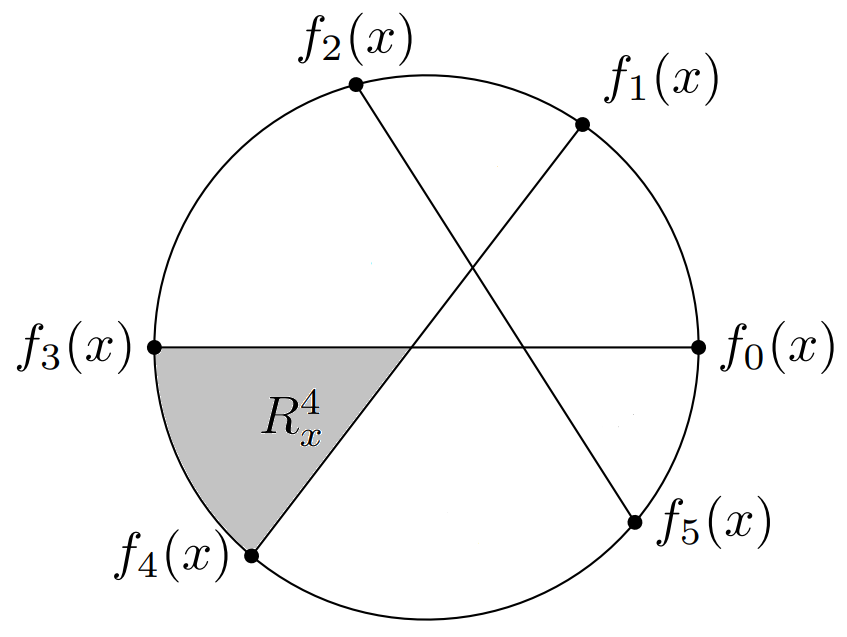}
    \includegraphics[scale=.25]{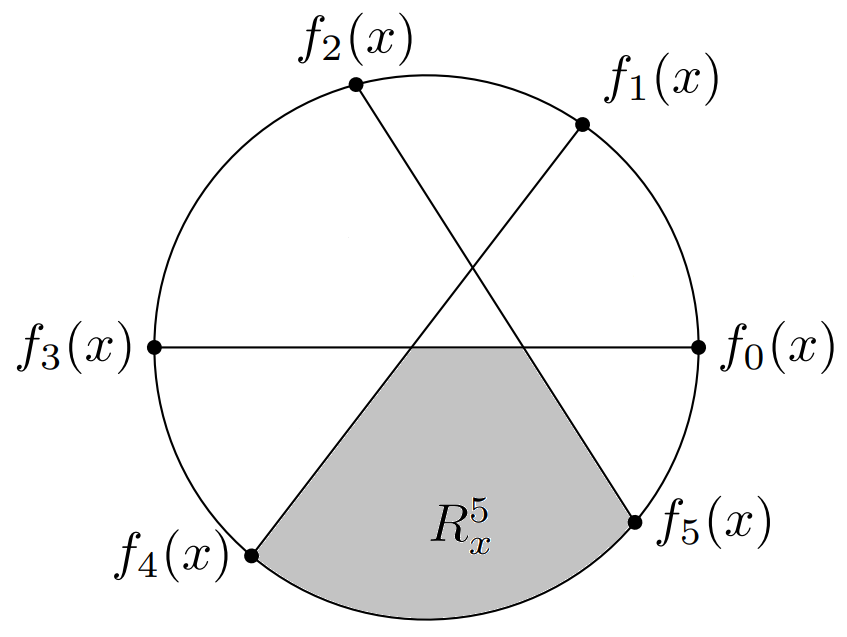}
    \caption{The regions $R^4_x$ (on the left) and $R^5_x$ (on the right).}
    \label{fig:R4R5}
\end{figure}

\begin{figure}
    \centering
    \includegraphics[scale=.25]{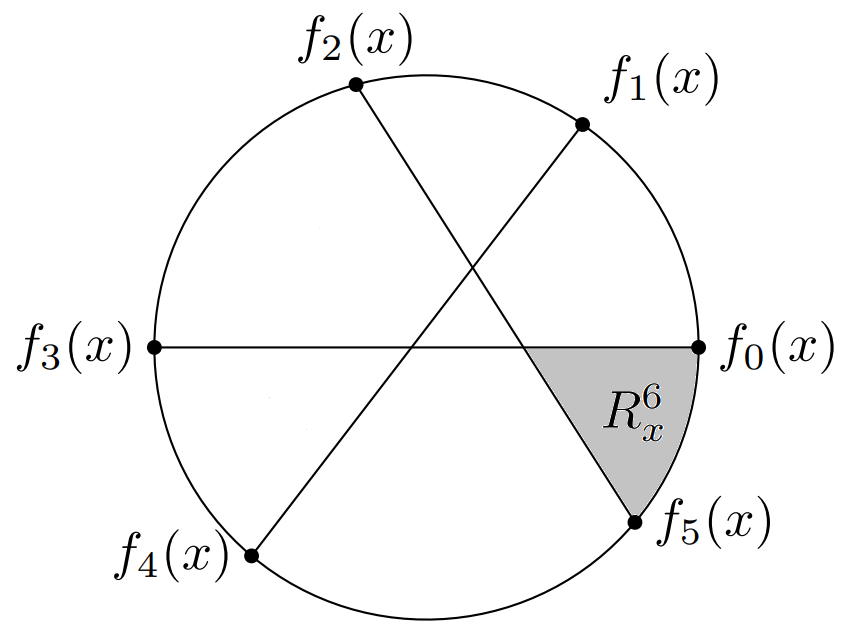}
    \caption{The region $R^6_x$.}
    \label{fig:R6}
\end{figure}

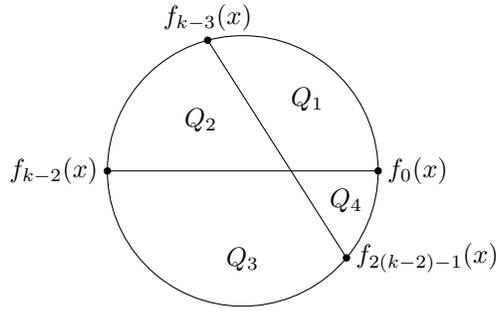
\begin{figure}
    \centering
    \begin{tikzpicture}[scale=.6,rotate=-40]
    \draw (0,0) circle (3cm);
\filldraw [black] (3,0) circle (2pt) node[right] {$f_{2(k-2)-1}(x)$};
\filldraw [black] (145:3cm) circle (2pt) node[above] {$f_{k-3}(x)$};
\filldraw [black] (40:3cm) circle (2pt) node[right] {$f_0(x)$};
\filldraw [black] (220:3cm) circle (2pt) node[left] {$f_{k-2}(x)$};

\draw (3,0) -- (145:3cm);
\draw (40:3cm) -- (220:3cm);

\filldraw [] (310:1.5cm) node[below] {$Q_3$};
\filldraw [] (35:2.3cm) node[below] {$Q_4$};
\filldraw [] (160:1.85cm) node[below] {$Q_2$};
\filldraw [] (95:2.5cm) node[below] {$Q_1$};

\end{tikzpicture}
    \caption{The quadrants $Q_1$, $Q_2$, $Q_3$, and $Q_4$.}
    \label{fig:quadrants}
\end{figure}

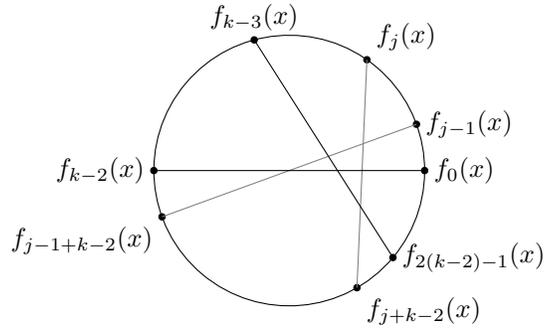
\begin{figure}
    \centering
    \begin{tikzpicture}[scale=.6,rotate=-40]
    \draw (0,0) circle (3cm);
\filldraw [black] (3,0) circle (2pt) node[right] {$f_{2(k-2)-1}(x)$};
\filldraw [black] (145:3cm) circle (2pt) node[above] {$f_{k-3}(x)$};
\filldraw [black] (40:3cm) circle (2pt) node[right] {$f_0(x)$};
\filldraw [black] (220:3cm) circle (2pt) node[left] {$f_{k-2}(x)$};

\filldraw [black] (95:3cm) circle (2pt) node[above right] {$f_j(x)$};
\filldraw [black] (60:3cm) circle (2pt) node[right] {$f_{j-1}(x)$};

\filldraw [black] (340:3cm) circle (2pt) node[below right] {$f_{j+k-2}(x)$};
\filldraw [black] (240:3cm) circle (2pt) node[below left] {$f_{j-1+k-2}(x)$};

\draw[gray] (95:3cm)--(340:3cm);
\draw[gray] (60:3cm)--(240:3cm);

\draw (3,0) -- (145:3cm);
\draw (40:3cm) -- (220:3cm);

\end{tikzpicture}
    \caption{If $\ell_j(x)$ intersects $Q_4$ and $\ell_{j-1}(x)$ intersects $Q_2$ (or the intersection of $\ell_0(x)$ and $\ell_{k-3}(x)$), then the intersection of $\ell_j(x)$ and $\ell_{j-1}(x)$ is in $Q_1$.}
    \label{fig:intersectQ1}
\end{figure}


\begin{thebibliography}{10}

\bibitem{eckhoff1969}
J\"{u}rgen Eckhoff.
\newblock Der {S}atz von {R}adon in konvexen {P}roduktstrukturen. {II}.
\newblock {\em Monatsh. Math.}, 73:7--30, 1969.

\bibitem{eckhofftransversal1973}
J\"{u}rgen Eckhoff.
\newblock Transversalenprobleme in der {E}bene.
\newblock {\em Arch. Math. (Basel)}, 24:195--202, 1973.

\bibitem{eckhoffgallai}
J\"{u}rgen Eckhoff.
\newblock A {G}allai-type transversal problem in the plane.
\newblock {\em Discrete Comput. Geom.}, 9(2):203--214, 1993.

\bibitem{eckhoffCommon2008}
J\"{u}rgen Eckhoff.
\newblock Common transversals in the plane: the fractional perspective.
\newblock {\em European J. Combin.}, 29(8):1872--1880, 2008.

\bibitem{holmsenNew2010}
Andreas~F. Holmsen.
\newblock New results for {$T(k)$}-families in the plane.
\newblock {\em Mathematika}, 56(1):86--92, 2010.

\bibitem{holmsenGeometric2013}
Andreas~F. Holmsen.
\newblock Geometric transversal theory: {$T(3)$}-families in the plane.
\newblock In {\em Geometry---intuitive, discrete, and convex}, volume~24 of
  {\em Bolyai Soc. Math. Stud.}, pages 187--203. J\'{a}nos Bolyai Math. Soc.,
  Budapest, 2013.

\bibitem{katchalskiSymmetric1980}
M.~Katchalski and A.~Liu.
\newblock Symmetric twins and common transversals.
\newblock {\em Pacific J. Math.}, 86(2):513--515, 1980.

\bibitem{knaster1929}
B.~Knaster, C.~Kuratowski, and S.~Mazurkiewicz.
\newblock Ein beweis des fixpunktsatzes f\"ur n-dimensionale simplexe.
\newblock {\em Fund. Math.}, 14(1):132--137, 1929.

\bibitem{Kramer} D. Kramer, Transversalenprobleme vom Hellyschen und Gallaischen Typ, Dissertation, Universit\"at Dortmund, 1974.

\bibitem{openkkm}
Marc Lassonde.
\newblock Sur le principe {KKM}.
\newblock {\em C. R. Acad. Sci. Paris S\'{e}r. I Math.}, 310(7):573--576, 1990.

\bibitem{mcginnis2021line}
Daniel McGinnis and Shira Zerbib.
\newblock Line transversals in families of connected sets the plane, 2021.

\bibitem{santalotheorem1940}
L.~A. Santal\'{o}.
\newblock A theorem on sets of parallelepipeds with parallel edges.
\newblock {\em Publ. Inst. Mat. Univ. Nac. Litoral}, 2:49--60, 1940.

\end{thebibliography}
\end{document}